\documentclass[a4paper,12pt]{article}
\usepackage{amssymb}
\usepackage{amsthm}
\usepackage[dvips]{graphicx}
\usepackage{amsmath}
\usepackage{fancyhdr}
\usepackage[all]{xy}

\theoremstyle{plain}
\newtheorem{theorem}{Theorem}[section]
\newtheorem{coro}[theorem]{Corollary}

\newtheorem{prop}[theorem]{Proposition}

\newtheorem{theo}[theorem]{Theorem}

\theoremstyle{definition}
\newtheorem{defi}[theorem]{Definition}
\newtheorem{rema}[theorem]{Remark}
\newtheorem{exam}[theorem]{Example}
\newtheorem*{notations}{\sc Notations}
\newtheorem*{acknowledgements}{\sc Acknowledgements}

\renewcommand{\baselinestretch}{1.13}

\newenvironment{reference}[1]{%

\begin{flushleft}\normalsize{\textsc{References}}\end{flushleft}%
\begin{enumerate}\setlength{\itemsep}{-5pt}\small
}{\end{enumerate}}

\makeatletter
\@addtoreset{equation}{section}
\makeatother

\makeatletter
\renewcommand{\section}{%
\@startsection{section}{1}{\z@}%
{3.5ex \@plus -1ex \@minus -.2ex}%
{2.3ex \@plus.2ex}%
{\reset@font\normalsize\scshape}}

\makeatother

\makeatletter
\renewcommand{\subsection}{%
\@startsection{subsection}{2}{\z@}%
{-3.5ex \@plus -1ex \@minus -.2ex}%
{-2.3ex \@plus.2ex}%
{\reset@font\normalsize\scshape}}
\makeatother

\usepackage[OT2,T1]{fontenc}
\DeclareSymbolFont{cyrletters}{OT2}{wncyr}{m}{n}
\DeclareMathSymbol{\Sha}{\mathalpha}{cyrletters}{"58}

\title{\vspace*{-22mm}
\large{\textbf{
Mild pro-$p$-groups and $p$-extensions of imaginary quadratic fields with non-trivial $p$-class group \\
}}
}

\author{
\textsc{\normalsize Zakariae Bouazzaoui}
 \and
\textsc{\normalsize Abdelaziz El Habibi}
}

\date{}

\begin{document}
{\renewcommand{\baselinestretch}{1.05} \maketitle}

\vspace*{-14mm}
\renewcommand{\abstractname}{}
{\renewcommand{\baselinestretch}{1.05}
\begin{abstract}{\small
\noindent\textsc{Abstract.}
Let $k$ be an imaginary quadratic field and $p$ an odd prime number such that the $p$-rank of the class group of $k$
is one. Let $S$ be a finite set of places of $k$ distinct from $p$-adic places. We give sufficient conditions for the Galois group $G_S$, of the maximal pro-$p$-extension of $k$ which is unramified outside $S$, to be \textit{mild}, hence of cohomological dimension $2$.
}\end{abstract}}

\section{\bf Introduction}

Let $p$ be a prime number, and let $k$ be a number field which is a finite extension of the field of rationals $\mathbb Q$. For a finite set $S$ of primes of $k$, consider the Galois group
\begin{equation*}
    G_{S}=\mathrm{Gal}(k_{S}/k)
\end{equation*}
of the maximal pro-$p$-extension $k_{S}$ of $k$ which is unramified outside $S$. Galois extensions with ramification restricted to finite sets of primes arise naturally in algebraic
number theory and arithmetic geometry, e.g., in terms of representations coming from
the Galois action on the (étale) cohomology of algebraic varieties defined over number fields. A non $p$-adic prime ideal $v$ of $k$ cannot ramify in any $p$-extensions over $k$ if the absolute norm $|\mathcal{O}_k/v| \not\equiv 1 \pmod{p}$,
where $\mathcal{O}_k$ is the ring of algebraic integers in $k$.
Since we consider only such pro-$p$-extensions over $k$ in the following,
we assume that
$S$ contains no archimedean primes if $p \neq 2$,
and that
$|\mathcal{O}_k/v| \equiv 1 \pmod{p}$ for any prime ideal $v\in S$ which is not $p$-adic. The structure of $G_{S}$ reflects arithmetic properties of the field $k$ at the prime $p$. In particular when $S=\emptyset$, the derived series of $\mathrm{Gal}(k_{\emptyset}/k)$ correspond to the $p$-class field tower of $k$, which is a classical
object in algebraic number theory. If $S$ contains the places of $k$ lying above $p$, the group $G_S$ is fairly well understood. In particular it has been known that it is of cohomological dimension $\mathrm{cd}(G_{S})$ is at most $2$ (assuming $k$ is totally imaginary if $p=2$) and
often a duality group (cf. \cite{NSW}).\\
So far the structure of $G_S$ remains mysterious when $S$ is disjoint with the $p$-adic places. By celebrate theorems of Golod and Shafarevich, the group $G_S$ is infinite under some condition and it is a so-called \textit{fab} pro-$p$-group, i.e., the maximal abelian quotient of every open subgroup of $G_S$ is finite. However, the cohomological dimension of $G_S$ is not well understood.\\
J. Labute introduce \textit{mild} pro-$p$-groups in \cite{Lab06} to give examples of groups $G_S$ with $\mathrm{cd}(G_S)=2$ when $S$ contains no $p$-adic places. These pro-$p$-groups are finitely generated and admits a presentation in terms of generators and defining relations, where the relations satisfy a condition of being free in a maximal possible way with respect to a suitable filtration. Equivalently, it is proved that $G_S$ is mild if the \textit{Poincaré series} associated to the filtration satisfies certain equality (e.g. \cite{Lab06}). Mild pro-$p$-groups are of cohomological dimension $2$, which is one of the main reasons for their importance also from an arithmetic point of view. Furthermore, these groups are not $p$-adic analytic. This should holds for $G_S$ as a consequence of \textit{tame Fontaine-Mazur} conjecture, which asserts that $G_S$ has no infinite $p$-adic analytic quotients (cf. \cite{Fontaine-Mazur}).
For the field $\mathbb{Q}$, Labute came up with the first examples of finite sets $S$ of primes of $\mathbb{Q}$ such that $S$ contains no $p$-adic places for an odd prime $p$ and  $\mathrm{cd}(G_{S})=2$. For this, the author proves that $G_{S}$ is mild for a suitable set of primes $S$ such that the associated linking diagram $\Gamma_{S}(p)$ forms a non-singular circuit in the sens of \cite[Definition 1.4]{Lab06}.
Later on, A. Schmidt \cite{Sch07}, \cite{Sch10} extended the result of Labute by arithmetic methods and weakened Labute’s condition on $S$. Shmidt's results uses a cup product criterion on the first cohomology groups that has shown its effectiveness in the problem of producing
mild pro-$p$-groups. In \cite{for11} Forr\'e gave simple and direct proof for the algebraic criteria for mildness. Here we use the approach based on Koch theory to study mildness \cite{Koch}.\\
The property of mildness has been extensively studied in the last decades. Examples of groups $G_S$ which are mild are known for the field of rationals $\mathbb{Q}$, and for imaginary quadratic fields not containing $p$-roots of unity with prime to $p$ class number \cite{Vog06}. We can also mention the papers \cite{Koch77}, \cite{Schmidt2006}, \cite{Bush-Labute}, \cite{Gar14} and \cite{Gar15}. Recently, in \cite{Minac21} Mina\v{c} et al. proved that if the
maximal pro-$p$ Galois group of a field is mild, then Positselski’s and Weigel’s Koszulity conjectures hold true for such a field.\vskip 4pt

The object of this paper is to study the property mild for the group $G_S$ in the case when $k$ is an imaginary quadratic number field with nontrivial $p$-class group, where $p$ is an odd prime number. By tracing Koch's theory, we first give in section $2$ a Koch type presentation by generators and relations for $G_S$ (Theorem \ref{minimal presentation}):
\[
  \xymatrix @C=2pc{1\ar@[>][r]&R\ar@[>][r]&F\ar@[>][r]^{\pi}&G_S\ar@[>][r]&1,}
\]
where $F$ is a free pro-$p$-group and $R$ is the relations group. This is done for certain set of primes of $k$ whose norms are congruent to $1$ modulo $p$, where the relations modulo the third step of the lower central series are described by linking numbers $\widetilde{l}_{*,*}$ of primes.\\
In section $3$ we describe an arithmetical sufficient condition for mildness of $G_S$ (Theorem \ref{mild}). These conditions are in terms of congruences modulo $p$ of liking numbers appearing in Theorem \ref{minimal presentation}. Furthermore,
in Corollary $3.3$ we give distinguished properties of the group $G_S$, for example we obtain that $G_S$ is a duality group with Euler characteristic $\chi(G_S)=1$. We obtain in Proposition \ref{4 generators mildness} a sufficient conditions for mildness of $G_S$ with $4$ generators.\\
The assumption of $p$-rank equals $1$ is to ease computations. However, by similar reasoning, we can obtain the same results in the case of fields for which the $p$-rank of the class group is less than $|S|$ (see Remark \ref{remark V_S}).\\

As an illustration, using PARI/GP we give the following example. Let $k$ be the quadratic field $\mathbb{Q}(\sqrt{-23})$ and $p=3$.
We consider the primes $\ell_{0}=13$, $\ell_{1}=211$, $\ell_{2}=67$ and
$\ell_{3}=31$. For suitable choices of the places $v_i$ above the primes $\ell_i$, we prove that $G_S$ is a mild pro-$p$-group for $S=\{v_{0},v_1,v_2,v_3\}$.

\begin{notations}
For a pro-$p$-group $G$, we denote by $[h,g]=h^{-1}g^{-1}hg$ the commutator of $g,h\in G$. For a closed subgroup $H$ of $G$, $[H,G]$ (resp. $H^{p^n}$) denote the minimal closed subgroup containing $\{[h,g]\,|\,h\in H, g\in G \}$ (resp. $\{h^{p^n} \,|\, h\in H \}$). Then $G^{ab}=G/[G,G]$. The $i$th cohomology group with coefficients in $\mathbb{F}_p=\mathbb{Z}/p\mathbb{Z}$ is denoted by $H^i(G)=H^i(G, \mathbb{Z}/p\mathbb{Z})$. For a set $Y$, $|Y|$ denote the cardinality. For ojects $x$ and $y$, $\delta_{x,y}=1$ if $x=y$ and $\delta_{x,y}=0$ otherwise.
\end{notations}

\section{\bf Koch type presentation}
Let $G$ be a finitely presented pro-$p$ group and let
\[
\entrymodifiers={+!!<0pt,\fontdimen22\textfont2>}
\xymatrix@R=0.8\baselineskip{
1 \ar[r] & R \ar[r] & F \ar[r] & G \ar[r] & 1
}
\]
be a minimal presentation, where $F$ is the free pro-$p$ group on the generators $x_1,\ldots, x_d$
and $R=\langle \rho_1,\ldots,\rho_r\rangle$ is the normal subgroup of $F$ generated by the elements $\rho_i$,
$i= 1,\ldots, r$.

\begin{defi}
The minimal presentation $\langle x_1,\ldots, x_d \,|\, \rho_1,\ldots,\rho_r \rangle$, of the pro-$p$-group $G$ is said to be of Koch type if $r\leq d$ and the relations $\rho_i$ satisfy a congruence of the form
\[ \rho_i\equiv
x_i^{pa_i}\prod_{i\neq j} [x_j,x_i]^{a_{ij}}
\mod{F_3},
\]
 with $a_i\in \mathbb Z_p$, $a_{ij}\in \mathbb F_p$, and $F_3$ is the third step of the descending $p$-central series of $F$. The group $G$ is of Koch type if it has a presentation of Koch type.
\end{defi}

\noindent Given a Koch type presentation of $G_S$ for a set of places $S$ of a number field $k$, one can derive
conditions for which the group $G_S$ is mild. Here we consider the following setting:

\begin{itemize}
    \item $p>2$ an odd prime number.
    \item $k$ an imaginary quadratic field such that the $\mathbb{F}_p$-rank of the class group $Cl_k$ is $1$.
\end{itemize}
In the following, for each place $v$ of $k$, we denote $k_v$ for the
completion of $k$ at $v$, and $U_v$ is the group of units of $k_v$.
Let $S$ be a finite set of primes of $k$ whose norms is congruent to
1 modulo $p$. We define the group
\begin{equation*}
V_{S}(k)=\{a\in k^{\times}\;|\;a\in (k_{v}^{\times})
^p\;\hbox{for}\; v\in S\;\hbox{and}\;a\in
U_{v}k_{v} \; \hbox{for} \; v\notin
S\}.
\end{equation*}
We can see that $(k^{\times})^p\subset V_{S}(k)$. Thus we can put
\begin{center}
 $B_S(k)=\big(V_{S}(k)/(k^{\times})^p\big)^{\ast}$,
\end{center}
where $(.)^{\ast}$ means the Pontryagin dual. We have the following exact
sequence

\[
  \xymatrix @C=2pc{0\ar@[>][r]& E_{k}/p\ar@[>][r]& V_{\emptyset}(k)/(k^{\times})^p\ar@[>][r]& Cl_k\ar@[>][r]^{p}&Cl_k,}
\]
where $E_k$ is the group of units of $k$.
We are interested in the sets $S$ of places for which
$V_{S}(k)=(k^{\times})^p$. This is intimately related to the structure
of the Galois group $G_S$. More
precisely, it gives information about the generators and relations
rank of $G_S$. Shafarevich's theorem \cite[Theorem
11.8]{Koch}, tells us that the minimal number $d(G_S)$ of generators
of $G_S$ (for our setting) is given by
\begin{eqnarray}\label{generators}
d(G_S)=\dim_{\mathbb{F}_p}H^1(G_S)=|S| + \dim_{\mathbb{F}_p}(B_S(k)).
\end{eqnarray}
Furthermore, Theorem $11.5$ of \cite{Koch} gives that
\begin{equation}\label{relations}
r(G_S)=\dim_{\mathbb{F}_p}H^2(G_S)\leq |S|+\dim_{\mathbb{F}_p}(B_{S}(k)).
\end{equation}

Let $\mathfrak{a}_1\not\in S$ be a prime ideal of $k$ such that $\{[\mathfrak{a}_1] (Cl_{k})^p\}$ is a basis of
$Cl_k/p$. Then
$Cl_k[p]$ is generated by $\{[\mathfrak{a}_{1}^{q_1/p}]\}$, where
$q_1$ denotes the order of $[\mathfrak{a}_1]$. In particular,
$\mathfrak{a}_{1}^{q_{1}}=a_{1}\mathcal{O}_{k}$ for some $a_{1}\in
V_{\emptyset}(k)$. The image of $a_{1}(k^{\times})^ p$, under the
homomorphism $V_{\emptyset}(k)/(k^{\times})^p \rightarrow Cl_k$, is
$ [\mathfrak{a}_{1}^{q_1/p}]$, thus we conclude that
\begin{eqnarray}\label{hyperprimery elements}
\{a_1 (k^{\times})^p\}\; \hbox{generates}\;
V_{\emptyset}(k)/(k^{\times})^p.
\end{eqnarray}

In the following, a set $S$ of places of
$k$ is said to be \textsf{singular} at a finite place $v_0$ if:
\begin{equation*}
v_0\in S\; \hbox{and} \;
a_{1}^{(|\mathcal{O}_{k}/v_0|-1)/p}\not\equiv 1 \pmod {v_0}.
\end{equation*}

\begin{prop}\label{B_S(K)}(\cite{El Habibi-Mizusawa})
Let $S$ be a finite set of places of $k$. Suppose that $S$ is
\textsf{singular} at $v_0$, then $B_S(k)=\{1\}$. It follows that
\begin{enumerate}
    \item  $d(G_S)=|S|$,
    \item  $r(G_S)\leq |S|$.
\end{enumerate}
\end{prop}

\begin{proof}
Since $V_{\emptyset}(k)=a_1^{\mathbb{Z}}(k^{\times})^p$ (see
(\ref{hyperprimery elements})) and $a_1\notin (k_{v_0}^{\times})^ p$ by
assumption, we have $a_1\notin
V_S(k)$, and hence $V_S(k)/(k^{\times})^p=\{1\}$. Thus $\mathrm{B}_S(k)=\{1\}$.\\
The assertion for minimal number of generators and relations of
$G_S$ follows from (\ref{generators}) and (\ref{relations}).
\end{proof}

We would like to give an explicit description of the set of
generators and relations in the group $G_S$ in our setting, while assuming  that $S$ is
\textsf{singular} and disjoint from the set of $p$-adic places of
$k$. For this we will need some local notations:\\
\noindent For any prime $v$ of $k$, we let $\pi_{v}$ be a uniformizer of $v$ and
$Frob_v$ is the corresponding Frobenius
automorphism in the maximal subextension of $k_S/k$ in which
a fixed prime of $k_S$ above $v$ is
unramified.\\
If $\Sigma$ is a set of places of $k$, we put
$$U_{\Sigma}:=\prod_{v\not\in\Sigma}U_{v}\times\prod_{v\in
\Sigma}\{1\}.$$

\noindent We set
$\gamma_{1}=(\gamma_{1,v})_v$ where $\gamma_{1,v}=\pi_{v}$ if
$v=\mathfrak{a}_1$ and $\gamma_{1,v}=1$ otherwise. Then
$\gamma_{1}Uk^{\times}$ correspond to the class $[\mathfrak{a}_1]$
under the isomorphism $J/Uk^{\times}\simeq Cl_k$, where
$U:=U_{\emptyset}$ and $J$ is the group of ideles of $k$. We have in particular
\begin{equation}\label{gamma_1}
\gamma_{1}^{-q_1}a_{1}\in U\;\;\hbox{and}\;\;q_{1}\equiv0\pmod{p}.
\end{equation}

Let $h$ be the order of the non-$p$-part of $Cl_k$. Let $w\in S$,
then $$[w]^{h}=[\mathfrak{a}_{1}]^{l_{w,1}},\;\;
l_{w,1}\in\mathbb{Z}.$$ Thus,
$w^{h}=\mathfrak{a}_{1}^{l_{w,1}}(\varpi_w)$ for some $\varpi_w\in
k^{\times}$. The idele element in $J$ corresponding to $\varpi_w$ is
given by $$(\varpi_{w})_v \in
\pi_{w}^{h}U_{w}\times\pi_{\mathfrak{a}_1}^{-l_{w,1}}U_{\mathfrak{a}_1}\times\prod_{v\not\in\{
w,\mathfrak{a}_1\}}U_{v}.$$

For $v\in S$, let $\alpha_v \in U_v$ be a primitive
$|(\mathcal{O}_{k}/v)^{\times}|$-th root of unity, and consider the
following elements of the group $G_S$:
\begin{itemize}
\item let $\sigma_{v}$ be an element of $G_S$ with
    the properties:
    \begin{enumerate}
\item $\sigma_{v}$ is a lift of
        $Frob_{v}$ to $G_S$,
\item
        $(\sigma_{v}^{h})_{|k_{S}^{ab}}=\left(\varpi_{v},k_{S}^{ab}/k\right)$,
\end{enumerate}
\item let $\tau_{v}$ be an element of $G_S$ such that
\begin{enumerate}
        \item $\tau_{v}$ is an element of the inertia group
        $\mathcal{T}_{v}$ of $v$ in $k_{S}/k$,
        \item
        $(\tau_{v})_{|k_{S}^{ab}}=\left(\alpha_{v},k_{S}^{ab}/k\right)$,
    \end{enumerate}
    \item let $\tau_1$ be a lift of $\left(\gamma_{1},k_{S}^{ab}/k\right)$ to
    $G_S$.
\end{itemize}

\begin{theo}\label{minimal presentation}
Let $S$ be a finite set of places of $k$. Suppose that $S$ is
\textsf{singular} at $v_0$. Then $G_S$ has a minimal presentation
\[
  \xymatrix @C=2pc{1\ar@[>][r]&R\ar@[>][r]&F\ar@[>][r]^{\pi}&G_S\ar@[>][r]&1,}
\]
where $F$ is a free pro-$p$-group with generators
$\{x_{\bf{v}}\}_{\mathbf{v}\in (S\setminus\{v_0\})\cup\{1\}}$ such
that $\pi(x_{\bf{v}})=\tau_{\mathbf{v}}$ for $\mathbf{v}\in
(S\setminus\{v_0\})\cup\{1\}$, and $R$ is a normal subgroup of $F$
normally generated by $\{\rho_{v}\}_{v\in S}$ of the form
\begin{equation}
\rho_{v}=x_{v}^{|\mathcal{O}_{k}/v|-1}[x_{v}^{-1},y_{v}^{-1}],\;\;
v\in S,
\end{equation}
where $y_{v}\in F$ such that $\pi(y_v)=\sigma_v$ and

\begin{equation*}
y_{v}\equiv \prod_{\mathbf{v} \in
(S\setminus\{v_0\})\cup\{1\}}x_{\bf{v}}^{\widetilde{l}_{v,\bf{v}}}\pmod{[F,F]},
\end{equation*}
where, for every $\mathbf{v} \in (S\setminus\{v_0\})\cup\{1\}$,
$\widetilde{l}_{v,\bf{v}}$ is some $p$-adic integer.

\end{theo}

\begin{proof}
Recall the following exact sequence from class field theory:
\begin{align}\label{eq:VVUGC}
0 \rightarrow V_S/(k^{\times})^p \rightarrow V_{\emptyset}/(k^{\times})^p \rightarrow U/(U)^pU_S \rightarrow (G_S)^{\mathrm{ab}}/p \rightarrow Cl_k/p \rightarrow 0.
\end{align}
It follows that the set $C=\left\{\tau_{1},\tau_{v},v\in S\right\}$ is a system of
generators of $G_{S}$. By Proposition \ref{B_S(K)}, a minimal system of generators has cardinality $|S|$, thus
we have to omit one
element of the set $C$ to get a minimal system of generators.\\
Note that by Hasse principal we have $k^{\times}\cap J^{p}=(k^{\times})^p$. Thus, $U\cap J^{p}=U^{p}$. We get the following isomorphisms
\begin{equation*}
V/(k^{\times})^p\simeq VJ^{p}/J^{p}=(k^{\times}J^{p}\cap
U)J^{p}/J^{p}\simeq (k^{\times}J^{p}\cap U)U^{p}/U^{p},
\end{equation*}
where $V=V_{\emptyset}$. By (\ref{gamma_1}) the maps are given by
\[
  \xymatrix @C=2pc{a_{1}(k^{\times})^p\ar@{|->}[r]& a_{1}J^{ p}=\gamma_{1}^{-q_1} a_{1}J^{p}\ar@{|->}[r]& \gamma_{1}^{-q_1} a_{1}U^{p}.}
\]
Then, the set $\{\gamma_{1}^{-q_1} a_{1}U^{p}U_{S}\}$ is a basis of
the image of the injective homomorphism \[
  \xymatrix @C=2pc{V/(k^{\times})^p\ar@{->}[r]& U/U^{p}U_{S}.}\]

For $v,w\in S$ such that $v\neq w,$ there are $p$-adic integers
$z_{1,v}$, $l_{w,v}$ satisfying
\begin{eqnarray}
\gamma_{1}^{-q_1} a_{1}&\equiv&\alpha_{v}^{z_{1,v}}\pmod{v},\\
\varpi_{w}^{-1}&\equiv&\alpha_{v}^{l_{w,v}}\pmod{v}.
\end{eqnarray}
We put $l_{w,v}=0$ if we have $v=w$.\vskip 6pt

The set $\{\alpha_{v}U^{p}U_{S}\}$ forms a basis of $U/U^{p}U_S$.
Then we have
\begin{equation*}
\prod_{v\in S}\alpha_{v}^{z_{1,v}}\equiv(\gamma_{1,v}^{-q_1}
a_{1})_v\pmod{U^{p}U_S},
\end{equation*}
hence
\begin{equation*}
\prod_{v\in
S}\alpha_{v}^{z_{1,v}}\equiv\gamma_{1}^{-q_1}\pmod{J^{p}k^{\times}U_S}.
\end{equation*}
Then, by class field theory we obtain
\begin{equation}\label{tau-tau}
\prod_{v\in
S}\tau_{v}^{z_{1,v}}\equiv\tau_{1}^{-q_1}\pmod{[G_{S},G_{S}]}.
\end{equation}
Recall that $\gamma_{1,v}=1$ if $v\neq\mathfrak{a}_1$. Since
$\gamma_{1,v}^{-q_{1}}a_{1}=a_{1}\equiv\alpha_{v}^{z_{1,v}}\pmod{v}$
for every $v\in S$, we have $z_{1,v_0}\not\equiv0\pmod{p}$ because
$S$ is singular at $v_0$. Thus a minimal system of generators of
$G_{S}$ is given by $C_{0}=\left\{\tau_{1},\tau_{v},v\in
S\setminus\{v_0\}\right\}$. In particular, using (\ref{tau-tau}) we
have
\begin{equation}\label{tau_v0}
\tau_{v_0}\equiv\tau_{1}^{-q_{1}z_{1,v_0}^{-1}}\prod_{v\in
S\setminus\{v_0\}}\tau_{v}^{-z_{1,v}z_{1,v_0}^{-1}}\pmod{[G_{S},G_{S}]}.
\end{equation}
For any $w\in S$ we have
\begin{eqnarray*}
\varpi_{w}=(\varpi_{w},(1)_{v\neq
w})&\equiv&(1,(\varpi_{w}^{-1})_{v\neq w})\pmod{k^{\times}}\\
&\equiv&(\gamma_{1}^{l_{w,1}})((1)_{v\not\in
S\setminus\{w\}},(\varpi_{w}^{-1})_{v\in
S\setminus\{w\}})\pmod{U_{S}J^{p}k^{\times}}\\
&\equiv&\gamma_{1}^{l_{w,1}}\prod_{v\in
S}\alpha_{v}^{l_{w,v}}\pmod{U_{S}J^{p}k^{\times}}.
\end{eqnarray*}
By class field theory we obtain
\begin{equation}\label{sigma_w^h}
\sigma_{w}^{h}\equiv\tau_{1}^{l_{w,1}}\prod_{v\in
S}\tau_{v}^{l_{w,v}}\pmod{[G_{S},G_S]}.
\end{equation}
Since $h\in\mathbb{Z}\setminus p\mathbb{Z}$, we have from
$(\ref{sigma_w^h})$:
\begin{equation}\label{sigma_w}
\sigma_{w}\equiv\tau_{1}^{l_{w,1}h^{-1}}\prod_{v\in
S}\tau_{v}^{l_{w,v}h^{-1}}\pmod{[G_{S},G_S]},
\end{equation}
where $h^{-1}$ is the inverse of $h$ modulo $p$. Recall that
$C_{0}=\left\{\tau_{1},\tau_{v},v\in S\setminus\{v_0\}\right\}$ is a
minimal system of generators of $G_S$. Let $F$ be a free
pro-$p$-group on $d=|S|$ generators. We define the surjective
homomorphism $\pi : F\rightarrow G_S$ by
$\pi(x_{\bf{v}})=\tau_{\mathbf{v}}$ for $\mathbf{v}\in
(S\setminus\{v_0\})\cup\{1\}$ and let $R$ be $\ker{(\pi)}$. For each
$w\in S$, there exists $y_{w}\in\pi^{-1}(\sigma_{w})$ such that by
(\ref{sigma_w}) we have
\begin{equation*}
y_{w}\equiv x_{1}^{l_{w,1}h^{-1}}\prod_{v \in
S}x_{v}^{l_{w,v}h^{-1}}\pmod{[F,F]}.
\end{equation*}
Moreover, by (\ref{tau_v0}) we have
\begin{equation}
x_{v_0}\equiv x_{1}^{-q_{1}z_{1,v_0}^{-1}}\prod_{v\in
S\setminus\{v_0\}}x_{v}^{-z_{1,v}z_{1,v_0}^{-1}}\pmod{[F,F]}.
\end{equation}
Thus, we find the congruence
\begin{equation*}
y_{w}\equiv \prod_{\mathbf{v} \in
(S\setminus\{v_0\})\cup\{1\}}x_{\bf{v}}^{\widetilde{l}_{w,\bf{v}}}\pmod{[F,F]},
\end{equation*}
where
\begin{equation*}
\widetilde{l}_{w,\bf{v}}=(l_{w,\bf{v}}-z_{1,\bf{v}}z_{1,v_0}^{-1}l_{w,v_0})h^{-1},\;
\hbox{for every}\;\mathbf{v} \in (S\setminus\{v_0\})\cup\{1\}, \;
\hbox{such that}\; z_{1, 1}=q_1.
\end{equation*}

Let $v\in S$, we have a natural homomorphism $G_{v}\rightarrow
G_{S}$ where $G_v$ is the Galois group of the compositum of all
$p$-extensions of $k_v$. Let $F_v$ be a free pro-$p$-group on $2$
generators $s_v$ and $t_v$. Define the homomorphism $\phi_{v} :
F_{v}\rightarrow F$ by $\phi_{v}(s_v)=y_v$ and $\phi_{v}(t_v)=x_v$.
Then we have a surjective homomorphism $F_{v}\rightarrow G_{v}$
whose kernel is the normal subgroup $R_v$ generated by
\begin{equation}
r_{v}=t_{v}^{|\mathcal{O}_{k}/v|-1}[t_{v}^{-1},s_{v}^{-1}].
\end{equation}
We have the following commutative diagram
\[
  \xymatrix@C=2pc{1\ar@[>][r]&R\ar@[>][r]&F\ar@[>][r]^{\pi}&G_S\ar@[>][r]&1\\
  1\ar@[>][r]& R_v \ar@[>][r]\ar@{->}[u]& F_v \ar@[>][r]\ar@{->}[u]^{\phi_{v}}& G_v\ar@[>][r]\ar@{->}[u]&1}
\]

By assumption we have $B_{S}(k)=\{1\}$, the morphism
\[
  \xymatrix@C=2pc{loc : H^{2}(G_S)\ar@[>][r]& \bigoplus_{v\in S} H^{2}(G_v)}
\]
is injective. Then $R$ is generated by $\{\phi_{v}(r_v)\}_{v\in S}$
(see \cite[Theorem 6.14]{Koch}). Which complete the proof of the
theorem.
\end{proof}

\begin{rema}\label{remark V_S}
The vanishing of the quotient $V_S(k)/(k^{\times})^p$ is an important condition in the proof of the above result.
However, this is not always true.
For example if $\mathrm{dim}_{\mathbb{F}_p}(Cl_k/p)>|S|$ then $V_S(k)/(k^{\times})^p\neq\{1\}$. Indeed, if
$V_S(k)/(k^{\times})^p=\{1\}$ then we have an injection

\[
  \xymatrix@C=2pc{ V_{\emptyset}/(k^{\times})^p \ar@[>][r]& U/(U)^pU_S}
\]

However, we have $U/(U)^pU_{S}\cong\prod_{v\in S}U_v/(U_v)^p $ and for any $v\in S$ we have
$\mathrm{dim}_{\mathbb{F}_p}(U_v/(U_v)^p)=1$. Thus, $\mathrm{dim}_{\mathbb{F}_p}(V_{\emptyset}(k)/(k^{\times})^p)=\mathrm{dim}_{\mathbb{F}_p}(Cl_k/p)\leq |S|$, which is absurd by assumption.
\end{rema}

In the following section, we use the above theorem to give sufficient conditions for mildness of the group $G_S$.

\section{\bf Mild pro-$p$-groups}

Let us recall what is Zassenhaus filtration for pro-$p$ groups. Let $G$ be such a group and put $G_{1}=G$ and $G_{n}=[G_{n-1},G]$ for $2\leq n$ recursively.
Then $\{G_n\}_{n\geq1}$ is the lower central series of $G$. Put $G_{(n)}=\{g\in G :\; g-1\in (I_G)^n\}$ for $n\geq1$, where $I_{G}=\mathrm{Ker}(\mathbb{F}_{p}[[G]]\rightarrow\mathbb{F}_p)$ is the augmentation ideal of $\mathbb{F}_{p}[[G]]$. Then $G_{(1)}=G$
and $\{G_{(n)}\}_{n\geq1}$ is the Zassenhaus filtration of $G$.\\
A finitely presented pro-$p$ group $G$ is said to be mild (with respect to the Zassenhaus filtration) when $G$ has a presentation $F/R\cong G$
with a system of relations which make a strongly free sequence in the graded Lie algebra $\mathrm{gr}(F)=\bigoplus_{n\geq1}F_{(n)}/F_{(n+1)}$ (see \cite{Lab06}). In this section we give a condition for mildness of the pro-$p$ groups $G_S=\mathrm{Gal}(k_{S}/k)$, where $S$ is a singular set of places of $k$.\vskip 6pt

Let $S=\{v_{0},...,v_{d-1}\}$ be a finite set of finite places of $k$ such that $|\mathcal{O}_k/v_i|\equiv1\pmod{p}$ for every $i\in\{0,...,d-1\}$. Assume that $S$ is singular at the place $v_{0}$. Let $I:=(S\setminus\{v_{0}\})\cup\{1\}$ and
recall that by Theorem \ref{minimal presentation},  $G_S \simeq F/R$ has generators $\{x_{\mathbf{v}_i}\}_{\mathbf{v}_i\in I}$ and relations $\{\rho_{v_i}\}_{1 \le i \le d-1}$.
For $f \in F$ and a multi-index $(\mathbf{v}_1\mathbf{v}_2 \cdots \mathbf{v}_n)$ with $\mathbf{v}_i \in I$,
we put
\[
\varepsilon_{(\mathbf{v}_1\mathbf{v}_2 \cdots \mathbf{v}_n),p}(f)
=\varepsilon_{\mathbb Z_p[[F]]}\left(\frac{\partial^n f}{
\partial x_{\mathbf{v}_1}\partial x_{\mathbf{v}_2} \cdots \partial x_{\mathbf{v}_d}
}\right) \bmod{p} \in \mathbb F_p ,
\]
which is called mod $p$ Magnus coefficient, where $\varepsilon_{\mathbb Z_p[[F]]} : \mathbb Z_p[[F]] \rightarrow \mathbb Z_p$ is the augmentation map,
and $\frac{\partial}{\partial x_{\mathbf{v}}}$ denotes the pro-$p$ Fox derivative which is a continuous endomorphism of $\mathbb Z_p[[F]]$ (\cite[\S 2]{Iha86}).

\begin{theo}\label{mild}
Let $S$ be \textsf{singular} at $v_0$ and of cardinality $4\leq
d\in2\mathbb{Z}$. Assume that $(S\setminus\{v_0\})\cup\{1\}$ is a
\textsf{circular set}, i.e. there exists a bijection
$\boldsymbol{v}:\mathbb{Z}/d\mathbb{Z}\rightarrow
(S\setminus\{v_0\})\cup\{1\}$ such that $\boldsymbol{v}(0)=1$ and
which satisfies the following conditions:
\begin{enumerate}
    \item $l_{v_{0},\boldsymbol{v}(d-2)}\equiv l_{v_{0},\boldsymbol{v}(2)}\equiv0\pmod{p}$\; \hbox{and}\;
    $z_{1,\boldsymbol{v}(j)}\equiv0\pmod{p}$\;\;\\
    \hbox{for all}\;\;$2\leq j\leq
    d-2$.
    \item
    $l_{\boldsymbol{v}(2i),\boldsymbol{v}(2j)}\equiv0\pmod{p}$\;\;\hbox{for
    any}\; $i,j$.
    \item $z_{1,\boldsymbol{v}(1)}l_{v_{0},1}\prod_{i=1}^{d-1}l_{\boldsymbol{v}(i),\boldsymbol{v}(i+1)}\not\equiv
    z_{1,\boldsymbol{v}(d-1)}l_{v_{0},1}\prod_{i=1}^{d-1}l_{\boldsymbol{v}(i),\boldsymbol{v}(i-1)}\pmod{p}$.
\end{enumerate}
Then $G_S$ is a mild pro-$p$-group of deficiency zero and hence of cohomological
dimension $2$.
\end{theo}

\begin{proof}
To establish the result we use a criterion based on the cup-product
(see \cite{Gar15}) yields that the pro-$p$-group $G_S$ is mild if the $\mathbb F_p$-vector space $H^1(G_S)$ has a decomposition $H^1(G_S)=U \oplus V$ with the subspaces $U$ and $V$ such that
\begin{center}
    $U\cup V=H^2(G_S)$ \,\,\, and \,\,\, $V\cup V=\{0\}$,
\end{center}
where $U\cup V$ (resp. $V\cup V$) denote the image of $U\otimes V$ (resp. $V\otimes V$) by the cup product
\[
\cup : H^1(G_S) \times H^1(G_S) \rightarrow H^2(G_S).
\]

By Theorem \ref{minimal presentation} above, the
group $G_S$ has a minimal presentation
\[
  \xymatrix @C=2pc{1\ar@[>][r]&R\ar@[>][r]&F\ar@[>][r]^{\pi}&G_S\ar@[>][r]&1,}
\]
which we use to define the trace map associated to $\rho \in R$:
$$\mathrm{tr}_{\rho} : H^2(G_S)
\stackrel{\mathrm{tg}^{-1}}{\longrightarrow} H^1(R)^{F/R}
\rightarrow \mathbb F_p : \psi \mapsto
\mathrm{tg}^{-1}(\psi)(\rho),$$ where $\mathrm{tg}$ is the
transgression isomorphism. The dual group $H^2(G_S)^{\vee}$ is
generated by $\{\mathrm{tr}_{\rho_v}\}_{v \in S}$ as an $\mathbb
F_p$-vector space. Let $\{\chi_{\mathbf{v}}\}_{\mathbf{v} \in
(S\setminus\{v_0\})\cup\{1\}}$ be the dual basis of $H^1(G_S) \simeq
H^1(F)$ such that
$\chi_{\mathbf{v}_1}(\pi(x_{\mathbf{v}_2}))=\delta_{\mathbf{v}_1,\mathbf{v}_2}$.
Note that $4 \le d \in 2\mathbb Z$. Then $H^1(G_S)=U \oplus V$,
where $U=\bigoplus_{i=0}^{(d-2)/2} \mathbb F_p
\chi_{\boldsymbol{v}(2i+1)}$ and $V=\bigoplus_{i=0}^{(d-2)/2}
\mathbb F_p \chi_{\boldsymbol{v}(2i)}$.\\
For $v \in S$
and $\mathbf{v}_1,\mathbf{v}_2 \in (S\setminus\{v_0\})\cup\{1\}$, we have
\begin{alignat*}{2}
\mathrm{tr}_{\rho_v}(\chi_{\mathbf{v}_1} \cup \chi_{\mathbf{v}_2})
&=&&-\varepsilon_{(\mathbf{v}_1 \mathbf{v}_2),p}(\rho_v)
=-\varepsilon_{(\mathbf{v}_1 \mathbf{v}_2),p}([x_v^{-1}, y_v^{-1}]) \\
&=&&-\varepsilon_{(\mathbf{v}_1),p}(x_v)\varepsilon_{(\mathbf{v}_2),p}(y_v)
+\varepsilon_{(\mathbf{v}_2),p}(x_v)\varepsilon_{(\mathbf{v}_1),p}(y_v) \\
&=&&-\delta_{v,\mathbf{v}_1}\varepsilon_{(\mathbf{v}_2),p}(y_{\mathbf{v}_1})
+\delta_{v,\mathbf{v}_2}\varepsilon_{(\mathbf{v}_1),p}(y_{\mathbf{v}_2}) \\
& &&+\delta_{v,v_0}\big(
-\varepsilon_{(\mathbf{v}_1),p}(x_{v_0})\varepsilon_{(\mathbf{v}_2),p}(y_{v_0})
+\varepsilon_{(\mathbf{v}_2),p}(x_{v_0})\varepsilon_{(\mathbf{v}_1),p}(y_{v_0})
\big)\\
&=&&-\delta_{v,\mathbf{v}_1}\widetilde{l}_{\mathbf{v}_1,\mathbf{v}_2}
+\delta_{v,\mathbf{v}_2}\widetilde{l}_{\mathbf{v}_2,\mathbf{v}_1}
+\delta_{v,v_0}\big(
z_{1,\mathbf{v}_1}\widetilde{l}_{v_0,\mathbf{v}_2}
-z_{1,\mathbf{v}_2}\widetilde{l}_{v_0,\mathbf{v}_1}
\big)z_{1,v_0}^{-1}
\end{alignat*}
by the basic properties of mod $p$ Magnus coefficients (cf. e.g., \cite[Theorem 2.4]{Gar14} and \cite[Proposition
2.18]{Vog05}).\vskip 6pt
\noindent By (1) and (2), we obtain
$\mathrm{tr}_{\rho_v}(\chi_{\boldsymbol{v}(2i)} \cup
\chi_{\boldsymbol{v}(2j)})=0$ for any $i,j,v$. Since the pairing
$$ H^2(G_S) \times
H^2(G_S)^{\vee} \rightarrow \mathbb{F}_p$$ is non-degenerate, $V
\cup V=\{0\}$.\vskip 6pt

Put a $d \times d$ matrix $\mathbf{A}=(a_{i,j})_{0
\le i,j \le d-1}$ with entries in $\mathbb F_p$ such that
\begin{align*}
a_{i,j}=\left\{\begin{array}{ll}
\mathrm{tr}_{\rho_{\boldsymbol{v}(i)}}(\chi_{\boldsymbol{v}(j)} \cup
\chi_{\boldsymbol{v}(j+1)})=-\delta_{\boldsymbol{v}(i),\boldsymbol{v}(j)}\widetilde{l}_{\boldsymbol{v}(j),\boldsymbol{v}(j+1)}
+\delta_{\boldsymbol{v}(i),\boldsymbol{v}(j+1)}\widetilde{l}_{\boldsymbol{v}(j+1),\boldsymbol{v}(j)}
& \hbox{if $i \neq 0$,}
\\
\mathrm{tr}_{\rho_{v_0}}(\chi_{\boldsymbol{v}(j)} \cup
\chi_{\boldsymbol{v}(j+1)})\;=\big(
z_{1,\boldsymbol{v}(j)}\widetilde{l}_{v_0,\boldsymbol{v}(j+1)}
-z_{1,\boldsymbol{v}(j+1)}\widetilde{l}_{v_0,\boldsymbol{v}(j)}
\big)z_{1,v_0}^{-1} & \hbox{if $i=0$.}
\end{array}\right.
\end{align*}

Suppose that $\sum_{j=0}^{d-1} c_j (\chi_{\boldsymbol{v}(j)} \cup
\chi_{\boldsymbol{v}(j+1)}) =\boldsymbol{0}$ in $U \cup V$. Since
\begin{align*}
a_{i,j}=\left\{\begin{array}{cl}
\mathrm{tr}_{\rho_{\boldsymbol{v}(i)}}(\chi_{\boldsymbol{v}(j)} \cup
\chi_{\boldsymbol{v}(j+1)}) & \hbox{if $1 \le i \le d-1$,}
\\
\mathrm{tr}_{\rho_{v_0}}(\chi_{\boldsymbol{v}(j)} \cup
\chi_{\boldsymbol{v}(j+1)}) & \hbox{if $i=0$,}
\end{array}\right.
\end{align*}
Under the conditions $(1)$ and $(3)$ we have
\begin{equation*}
\mathbf{A} =(a_{i,j})_{i,j}=h^{-d}\left(
\begin{array}{ccccc}
-z_{1,\boldsymbol{v}(1)}l_{v_{0},1}z_{1,v_{0}}^{-1}&    &    &    &  z_{1,\boldsymbol{v}(d-1)}l_{v_{0},1}z_{1,v_{0}}^{-1} \\
l_{\boldsymbol{v}(1),1} & -l_{\boldsymbol{v}(1),\boldsymbol{v}(2)} &   & & \\
        & l_{\boldsymbol{v}(2),\boldsymbol{v}(1)} & \ddots  &   &  \\
     &   &    \ddots     & -l_{\boldsymbol{v}(d-2),\boldsymbol{v}(d-1)}  \\
        &     &    & l_{\boldsymbol{v}(d-1),\boldsymbol{v}(d-2)}& -l_{\boldsymbol{v}(d-1),1}
\end{array}
\right)
\end{equation*}
and
$$\det{(\mathbf{A})}=z_{1,v_{0}}^{-1}h^{-d}\left(z_{1,\boldsymbol{v}(1)}l_{v_{0},1}\prod_{i=1}^{d-1}l_{\boldsymbol{v}(i),\boldsymbol{v}(i+1)}-
    z_{1,\boldsymbol{v}(d-1)}l_{v_{0},1}\prod_{i=1}^{d-1}l_{\boldsymbol{v}(i),\boldsymbol{v}(i-1)}\right).$$
By the condition (3), we obtain $\det{(\mathbf{A})}\not\equiv0\pmod{p}$.

We have
\[
(0,0,\cdots,0)=(\mathrm{tr}_{\rho_{v_0}}(\boldsymbol{0}),\mathrm{tr}_{\rho_{\boldsymbol{v}(1)}}(\boldsymbol{0}),\cdots,\mathrm{tr}_{\rho_{\boldsymbol{v}(d-1)}}
(\boldsymbol{0}))=(c_0,c_1,\cdots,c_{d-1})\,{^t\!\mathbf{A}},
\]
where ${^t\!\mathbf{A}}$ is the transpose of $\mathbf{A}$. Since
$\{\chi_{\boldsymbol{v}(j)} \cup \chi_{\boldsymbol{v}(j+1)}\}_{0 \le
j \le d-1}$ is linearly independent, then $U \cup V=H^2(G_S)$ and
Therefore $G_S$ is a mild pro-$p$ group. In particular, $\mathrm{cd}(G_S)=2$.
\end{proof}

\begin{rema}
The approach of this paper computes the cup product using Koch theory of minimal presentations
and $\pmod{p}$ Magnus coefficients. We note that another approach due to A. Schmidt gives a cup-product
criterion by computing local components of the cup-product (see Section $5$ of \cite{Sch10}).
\end{rema}

\begin{defi}
A pro-$p$-group $G$ is \textbf{fab} if for every open subgroup
$H\subseteq G$ the abelianization $H^{ab}=H/[H,H]$ is finite. We
call $G$ fabulous if it is mild and fab.
\end{defi}
\noindent Note that Labute gives a slightly different definition of what is a
fabulous pro-$p$-group, where the group must satisfy an additional
property called \textsl{quadratic} (see \cite[Definition 7]{Lab08}).

\begin{coro}
Under the conditions of Theorem \ref{mild}, the pro-$p$-group
$G_{S}$ satisfies the following properties:
\begin{enumerate}
    \item $G_S$ is fabulous.
    \item $G_S$ is a duality group.
    \item The strict cohomological dimension is $scd(G_S)=3$.
    \item $G_S$ is not $p$-adic analytic.
    \item $\chi(G_S)=1$ the Euler characteristic of $G_S$.
\end{enumerate}
\end{coro}

\begin{proof}

For the property (1), since $S$ is disjoint from the set $S_p$ of
places lying above $p$, the extension $k_{S}/k$ does not contain any
$\mathbb{Z}_p$ extensions, thus $G_{S}^{ab}$ is finite, hence
$G_{S}$ is fabulous.\\
The second and the third properties follows from Proposition $1.3$
of \cite{Wing07}, and $(4)$ follows from \cite[Theorem 2.7]{Gar14}.
\end{proof}

Let $S$ be a finite set of places of $k$ of cardinality $4$ such
that $S=\{v_{0},v_1,v_3,v_4\}$ where each $v_i$ is lying over a
prime number $\ell_i$ satisfying $\ell_{i}\equiv1\pmod{p}$.

\begin{prop}\label{4 generators mildness}
The group $G_S$ is a mild pro-$p$-group with $4$ generators if the
following conditions are satisfied:
\begin{enumerate}
\item  The places $v_0$, $v_1$ and $v_3$ have degree
one, and $\ell_2$ is inert in $k$.

\item $a_{1}^{\frac{\ell_{0}-1}{p}}\not\equiv1\pmod{v_0}$. 

\item
$\varpi_{v_0}^{\frac{\ell_{2}^{2}-1}{p}}\equiv1\pmod{\ell_2}$ and
$a_{1}^{\frac{\ell_{2}^{2}-1}{p}}\equiv1\pmod{\ell_2}$.

\item $a_{1}^{\frac{\ell_{1}-1}{p}}\not\equiv1\pmod{v_1}$,
$\varpi_{v_1}^{\frac{\ell_{2}^{2}-1}{p}}\not\equiv1\pmod{\ell_2}$
and $\varpi_{v_{2}}^{\frac{\ell_{3}-1}{p}}\not\equiv1\pmod{v_3}$.

\item In the $p$-Hilbert class field of
$k$, $v_0$ and $v_3$ are inert, and $v_1$ is split.
\end{enumerate}
\end{prop}

\begin{proof}
Let $\boldsymbol{v}:\mathbb{Z}/4\mathbb{Z}\rightarrow
(S\setminus\{v_0\})\cup\{1\}$ be a bijection such that
$\boldsymbol{v}(0)=1$ and $\boldsymbol{v}(i)=v_i$ for every
$i=1,2,3$. The property $(2)$ gives that $S$ is singular at $v_0$.
By $(3)$ we obtain $l_{v_{0},v_2}\equiv0\pmod{p}$ and
$z_{1,v_2}\equiv0\pmod{p}$. Further, since $\ell_2$ is inert, we
have $l_{v_{2},1}\equiv0\pmod{p}$. The place $v_1$ splits in the
$p$-Hilbert class field of $k$, then $l_{v_{1},1}\equiv0\pmod{p}$.
Thus we have
$$z_{1,\boldsymbol{v}(3)}l_{v_{0},1}\prod_{i=1}^{3}l_{\boldsymbol{v}(i),\boldsymbol{v}(i-1)}\equiv0\pmod{p}.$$
Moreover, by the properties $(4)$ and the fact that $v_0$ and $v_3$
are inert in the $p$-Hilbert class field of $k$,
$$z_{1,\boldsymbol{v}(1)}l_{v_{0},1}\prod_{i=1}^{3}l_{\boldsymbol{v}(i),\boldsymbol{v}(i+1)}\not\equiv0\pmod{p}.$$
Then Theorem \ref{mild} gives that $G_S$ is a mild pro-$p$-group.
\end{proof}

We end this section with the following computations performed using PARI/GP \cite{pari}.

\begin{exam}
Let $k$ be the quadratic field $\mathbb{Q}(\sqrt{-23})$ and $p=3$.
We consider the primes $\ell_{0}=13$, $\ell_{1}=211$, $\ell_{2}=67$ and
$\ell_{3}=31$. Let $\mathcal{O}_k$ be the ring of integers of $k$.
We have $Cl_{k}/p=\langle[\mathfrak{a}_{1}]\rangle$ where
$\mathfrak{a}_{1}=3\mathcal{O}_{k}+(\sqrt{-23}-1)\mathcal{O}_{k}$
and $v_{0}=13\mathcal{O}_{k}+(\sqrt{-23}-4)\mathcal{O}_{k}$ a prime
ideal over $\ell_0$. We set $v_{1}=\varpi_{v_1}\mathcal{O}_{k}$ a
prime ideal above $\ell_1$ such that $\varpi_{v_1}=-2+3\sqrt{-23}$
and we put $v_{2}=\ell_{2}\mathcal{O}_{k}$. We have
$a_{1}=-2-\sqrt{-23}$ such that
$\mathfrak{a}_{1}^{q_1}=a_{1}\mathcal{O}_{k}$ and we obtain
$a_{1}^{\frac{\ell_{0}-1}{p}}\equiv9\pmod{v_0}$,
$a_{1}^{\frac{\ell_{2}^{2}-1}{p}}\equiv1\pmod{\ell_2}$,
$a_{1}^{\frac{\ell_{1}-1}{p}}\equiv14\pmod{v_1}$. Moreover,
$\varpi_{v_0}^{\frac{\ell_{2}^{2}-1}{p}}\equiv1\pmod{\ell_2}$,
$\varpi_{v_1}^{\frac{\ell_{2}^{2}-1}{p}}\equiv37\pmod{\ell_2}$ and
$\varpi_{v_{2}}^{\frac{\ell_{3}-1}{p}}\equiv5\pmod{v_3}$, where
$v_{3}=31\mathcal{O}_{k}+(\sqrt{-23}-15)\mathcal{O}_{k}$ and
$\varpi_{v_{2}}=\ell_2$. Then Proposition \ref{4 generators
mildness} gives that $G_S$ is mild for $S=\{v_{0},v_1,v_2,v_3\}$, hence it is of cohomological dimension $2$.
\end{exam}

\textbf{Further comments}\\
If the linking numbers $\widetilde{l}_{*,*}$ in Theorem \ref{minimal presentation} all vanishes modulo $p$, then $F_{(3)}$ contains the relations module $R$. Thus the cup product is trivial. To study fabulous and mild pro-$p$-groups we need to compute relations modulo $F_{(4)}$.
Examples are known for the field of rationals $\mathbb{Q}$ by the work of Vogel, G\"artner and Maire (e.g. \cite{Vog05}, \cite{Gar14}, \cite{Maire14}). It would be of interest to extand these results to imaginary quadratic fields with non trivial $p$-class group.

\vspace*{15pt}
\begin{acknowledgements}
The authors thank Yasushi Mizusawa for reading an earlier version of this paper and for his valuable comments and suggestions.
We also thank Christian Maire for his interest in this work.
\end{acknowledgements}


\vspace*{10pt}

{\footnotesize
\noindent\textsc{Zakariae Bouazzaoui}\\
Department of Mathematics, Faculty of Sciences of Meknès, Morocco. \\
\texttt{z.bouazzaoui@edu.umi.ac.ma}
\\[2mm]%
\noindent\textsc{Abdelaziz El Habibi}\\
Department of Mathematics, Faculty of Sciences, Mohammed First
University, Oujda, Morocco. \\
\texttt{a.elhabibi@ump.ac.ma}
\\[2mm]%
}

\end{document}